\documentclass[12pt]{amsart} \usepackage{amssymb} \usepackage{mathabx}
\usepackage{graphicx}

\newcommand{\wrlab}[1]{\label{#1}}

\newtheorem{thm}{THEOREM}[section]

\newtheorem{lem}[thm]{LEMMA} 
\newtheorem{cor}[thm]{COROLLARY} 
 
 \newtheorem{thm*}{THEOREM}[]

\newcommand{\tref}[1]{Theorem~\ref{#1}}
\newcommand{\cref}[1]{Corollary~\ref{#1}}

\newcommand{\lref}[1]{Lemma~\ref{#1}}

\def\N{{\mathbb N}} \def\Z{{\mathbb Z}} \def\Q{{\mathbb Q}}
\def\R{{\mathbb R}}

 \def\scra{{\mathcal A}} 
 \def\scrf{{\mathcal F}}

  \def\scrt{{\mathcal T}}

\def\bfzero{{\bf 0}}

\font\tenolde=eufm10 at 10pt
\font\sevenolde=eufm7
\font\fiveolde=eufm5
\newfam\oldefam
\textfont\oldefam=\tenolde
\scriptfont\oldefam=\sevenolde
\scriptscriptfont\oldefam=\fiveolde

\def\dom{\hbox{dom\,}}
\def\im{\hbox{im\,}}

\def\exp{\hbox{exp}}

\def\Id{\hbox{Id}}

\begin{document}
\bibliographystyle{alpha}

\title[Real analytic counterexample]
{Real analytic counterexample to the freeness conjecture}


\author{Scot Adams}
\address{School of Mathematics\\ University of Minnesota\\Minneapolis, MN 55455
\\ adams@math.umn.edu}

\date{September 4, 2015\qquad Printout date: \today}

\begin{abstract}
  We provide a counterexample to P.~Olver's freeness conjecture for
  $C^\omega$ transformations.
\end{abstract}

\maketitle

 

\section{Introduction\wrlab{sect-intro}}

P.~Olver's freeness conjecture (in his words) asserts: ``If a Lie
group acts effectively on a manifold, then, for some $n<\infty$, the
action is free on [a nonempty] open subset of the jet bundle of order $n$.''
In this note, we provide a counterexample
to one interpretation of this conjecture:
In \tref{thm-ctrx}, we show that a $C^\omega$ counterexample exists
for the additive group $\Z$ of integers.
Using Lemma 17.3 of \cite{a1:cinftyfreeness},
and following the proof of Theorem 18.1 in \cite{a1:cinftyfreeness},
it is possible to induce that $\Z$-action to get a $C^\omega$~counterexample for any
connected Lie group with noncompact center.
Conversely, according to Theorem 19.1 of \cite{a1:cinftyfreeness},
the conjecture holds
(in its $C^\infty$ form and, consequently, in its $C^\omega$ form)
for any connected Lie group with compact center.

A $C^\infty$ counterexample for an action of $\Z$
can be found in the proof of Theorem 18.1 of \cite{a1:cinftyfreeness}.
In that counterexample, $\Z$ acts on $\R^4$, and it seems likely that
a $C^\omega$ counterexample could also be constructed on~$\R^4$, or, possibly,
even on a lower dimensional Euclidean space.
The best partial result I know of in this direction is \cite{morris:tosandp},
due to D.~Morris, and involves a $C^\omega$ action on $\R$
of the additive semigroup of positive integers.
However, the rigidity of $C^\omega$ actions makes this kind
of argument technically challenging.
In the proof of \tref{thm-ctrx},
we avoid many technical issues through the use of topology.
To wit, we construct a $C^\omega$~counterexample on a $3$-dimensional manifold
through a kind of dynamical surgery:
Following an iterative recipe, we form a subset $M$ of~$\R^3$,
and give it a locally Euclidean topology $\tau$.
This topology $\tau$ is chosen so as to reproduce the effect of gluing
certain pairs of submanifolds together,
and then passing to a fundamental domain.
The resulting topological space $(M,\tau)$
has infinitely generated fundamental group,
so is significantly more complicated, topologically,
than a contractible space like~$\R^4$.
Also constructed iteratively is
a maximal $C^\omega$ atlas $\scra$ on $(M,\tau)$,
along with a $C^\omega$ vector field $V$
on the $C^\omega$ manifold $(M,\tau,\scra)$.
The flow of $V$ is an $\R$-action that,
when restricted to $\Z$, provides the desired counterexample.

This writeup is not intended for publication.

\section{Miscellaneous notation and terminology\wrlab{sect-notation}}

Let $\N:=\{1,2,3,\ldots\}$.
For any sets $A,B$ and function $f:A\to B$,
let $\dom[f]:=A$ be the domain of $f$,
and let $\im[f]:=f(A)\subseteq B$ be the image of $f$.
For any set $A$, let $\Id_A:A\to A$ be the identity map on~$A$.

A subset of a topological space is {\bf meager}
(a.k.a.~{\bf of first category})
if it is a countable union of nowhere dense sets.
A subset of a topological space is {\bf nonmeager}
(a.k.a.~{\bf of second category}) if it is not meager.
A subset of a topological space is {\bf comeager}
(a.k.a.~{\bf residual}) if its complement is meager.
A subset of a topological space is {\bf locally closed}
if it is the intersection of an open set with a closed set.
A subset of a topological space is {\bf constructible}
if it is a finite union of locally closed sets.
The collection of constructible sets is exactly the
Boolean algebra of sets generated from the topology.

Let $M$ be a set, let $\tau$ be a topology on $M$
and let $M_0$ be a $\tau$-open subset of $M$.
We define $\tau|M_0:=\{U\in\tau\,|\,U\subseteq M_0\}$.
For any maximal $C^\omega$ atlas $\scra$ on $(M,\tau)$, we define
$\scra|M_0:=\{\phi\in\scra\,|\,\dom[\phi]\subseteq M_0\}$.

Let $M$ be a $C^\omega$ manifold.
Let $V$ be a complete $C^\omega$ vector field on~$M$.
For all $t\in\R$, we denote the time $t$ flow of $V$ by $\Phi_t^V:M\to M$.
By the Cauchy-Kowalevski Theorem,
$(\sigma,t)\mapsto\Phi_t^V(\sigma):M\times\R\to M$ is $C^\omega$.
For any $A\subseteq\R$, for any $B\subseteq M$, let $\Phi_A^V(B):=\{\Phi_a^V(b)\,|\,a\in A,b\in B\}$.
For any $A\subseteq\R$, for any $b\in M$, let $\Phi_A^V(b):=\{\Phi_a^V(b)\,|\,a\in A\}$.
For any~$a\in\R$, for any $B\subseteq M$, let $\Phi_a^V(B):=\{\Phi_a^V(b)\,|\,b\in B\}$.

Let $M$ be a $C^\omega$ manifold, let $V$ be a $C^\omega$ vector field on $M$,
let $k\ge0$ be an integer and let $\sigma\in M$.
We say $(V,\sigma)$ is {\bf periodic to order $k$} if, for some integer $T\ne0$,
the map $\Phi_T^V:M\to M$ agrees with the identity map
$\Id_M:M\to M$ to order $k$ at $\sigma$.

Let $\tau_\#$ denote the standard topology on $\R^3$.
For all $S\subseteq\R^3$, let $S^\circ$~be the $(\tau_\#)$-interior in $\R^3$ of $S$.
Let $\scra_\#$ denote the standard maximal $C^\omega$~atlas on $(\R^3,\tau_\#)$.
Let $E$ be the $C^\omega$ vector field on $(\R^3,\tau_\#,\scra_\#)$ represented by the constant map
$(x,y,z)\mapsto(1,0,0):\R^3\to\R^3$.
Then, for all $x,y,z,t\in\R$, we have $\Phi_t^E(x,y,z)=(x+t,y,z)$.
Define functions $\pi:\R^3\to\R$ and $\Pi:\R^3\to\R^2$ by
$\pi(x,y,z)=x$ and $\Pi(x,y,z)=(y,z)$.

A topology will be called a {\bf manifold topology} if it is
Hausdorff, second countable and locally Euclidean.
For any integer $N\ge1$, let
\begin{gather*}
L_N\,:=\,(-\infty,-N+1)\,\times\,\R^2,\qquad
R_N\,:=\,(N-1,\infty)\,\times\,\R^2\\
\hbox{and}\qquad S_N\,:=\,(-N,N)\,\times\,\R^2.
\end{gather*}
A {\bf displayed system} consists of
\begin{itemize}
\item an integer $N\ge1$,
\item a $(\tau_\#)$-constructible subset $M$ of $\R^3$,
\item a manifold topology $\tau$ on $M$,
\item a maximal $C^\omega$ atlas $\scra$ on $(M,\tau)$ \qquad\qquad and
\item a complete $C^\omega$ vector field $V$ on $(M,\tau,\scra)$
\end{itemize}
such that
\begin{itemize}
\item$M^\circ$ is $\tau$-open and $\tau$-dense in $M$,
\item$\tau\,|\,(M^\circ)\,\,=\,\,(\tau_\#)\,|\,(M^\circ)$,
\item$\scra\,|\,(M^\circ)\,\,=\,\,(\scra_\#)\,|\,(M^\circ)$,
\item$M\cap S_N$ is $\tau$-open in $M$,
\item$L_N\,\cup\,R_N\,\,\subseteq\,\,M^\circ$ \qquad\qquad\qquad and
\item there exists a function $\theta:M^\circ\to[0,1]$ such that
 \begin{itemize}
 \item[]$\theta=1$ on $L_N\cup R_N$ \quad\qquad and \quad\qquad $V=\theta E$ on $M^\circ$.
 \end{itemize}
\end{itemize}

Let $D=(N,M,\tau,\scra,V)$ be a displayed system.
We say that $\sigma\in M$ is {\bf$D$-flat} if there exists $v\in\R^2$ such that
$(-N,v),(N,v)\in\Phi_\R^V(\sigma)$.
Let $\scrf(D)$ denote the set of $D$-flat points in $M$.
We say that $D$ is {\bf generically flat} if $\scrf(D)$ is $\tau$-comeager in $M$.

Let $D=(N,M,\tau,\scra,V)$ and $D'=(N',M',\tau',\scra',V')$ be displayed systems.
We say $D'$ is an {\bf extension of $D$} if all of the following hold:
\begin{itemize}
\item$N'\,\,\ge\,\,N+1$,
\item$M\cap S_N$ is $\tau'$-open in $M'$,
\item$M'\cap S_N\quad=\quad M\cap S_N$,
\item$\tau'\,|\,(M\cap S_N)\quad=\quad\tau\,|\,(M\cap S_N)$,
\item$\scra'\,|\,(M\cap S_N)\quad=\quad\scra\,|\,(M\cap S_N)$ \qquad\qquad and
\item$V'\,|\,(M\cap S_N)\quad=\quad V\,|\,(M\cap S_N)$.
\end{itemize}

\section{Coincidence of jets\wrlab{sect-facts-real-an}}

\begin{lem}\wrlab{lem-match-jet}
Let $d\ge1$ and $k\ge0$ be integers.
Let $w_0\in\R^d$ and let $W$~be an open neighborhood in $\R^d$ of $w_0$.
Let $\lambda:W\to\R$ be $C^k$
and assume that $2\le\lambda(w_0)\le3$.
Then there exists a $C^\omega$ function
$\lambda_0:\R^d\to(1,4)$
such that $\lambda_0$ agrees with $\lambda$ to order $k$ at $w_0$.
\end{lem}

\begin{proof}
Let $\bfzero:=(0,\ldots,0)\in\R^d$.
We assume, without loss, that $w_0=\bfzero$.
Define $\mu:=\lambda-2$.
Then $0\le\mu(\bfzero)\le1$.
Let $P:\R^d\to\R$ be a polynomial such that $P$ agrees with $\mu$ to order $k$ at $\bfzero$.
Define $Q:\R^d\to\R$ by $Q(x_1,\ldots,x_d)=x_1^{2k+2}+\ldots+x_d^{2k+2}$.
Then $Q\ge0$.
Also, $Q$ vanishes to order $2k+1$ at $\bfzero$.
For all $a>0$, let $\mu_a:=Pe^{-aQ}:\R^d\to\R$;
then $\mu_a$~agrees with $P$ to order $2k+1$ at $\bfzero$,
so $\mu_a$ agrees with $\mu$ to order $k$ at $\bfzero$.
We have $P(\bfzero)=\mu(\bfzero)$, so $0\le P(\bfzero)\le1$.
Choose $a_0>0$ so large that $-1<\mu_{a_0}<2$.
Let $\lambda_0:=2+\mu_{a_0}$.
\end{proof}

\section{Facts about displayed systems\wrlab{sect-facts-disp-syst}}

Let $D=(N,M,\tau,\scra,V)$ be a displayed system.

\begin{lem}\wrlab{lem-unambig-nw-dense}
Let $Z\subseteq M$. Then
\begin{itemize}
\item[]$Z$ is $\tau$-nowhere dense \qquad \hbox{iff} \qquad $Z$ is $(\tau_\#)$-nowhere dense.
\end{itemize}
\end{lem}

\begin{proof}
Let $Z'=Z\cap(M^\circ)$.
Then $\,\,Z'\,\,\subseteq\,\,Z\,\,\subseteq\,\,[Z']\cup[M\backslash(M^\circ)]$.

Since $M^\circ$~is $\tau$-open and $\tau$-dense in $M$,
it follows that $M\backslash(M^\circ)$ is $\tau$-nowhere dense in $M$.
Then
\begin{itemize}
\item[(a)]$Z$ is $\tau$-nowhere dense \qquad iff \qquad $Z'$ is $\tau$-nowhere dense.
\end{itemize}
Since $Z'\subseteq M^\circ$ and since $\tau|(M^\circ)=(\tau_\#)|(M^\circ)$, it follows that
\begin{itemize}
\item[(b)]$Z'$ is $\tau$-nowhere dense \qquad iff \qquad $Z'$ is $(\tau_\#)$-nowhere dense.
\end{itemize}
Since $M$ is a $(\tau_\#)$-constructible subset of $\R^3$,
it follows that $M\backslash(M^\circ)$~is $(\tau_\#)$-nowhere dense in $\R^3$.
Then
\begin{itemize}
\item[(c)]$Z$ is $(\tau_\#)$-nowhere dense \qquad iff \qquad $Z'$ is $(\tau_\#)$-nowhere dense.
\end{itemize}
The result now follows from (a), (b) and (c).
\end{proof}

\begin{cor}\wrlab{cor-unambig-nw-dense}
Let $Z\subseteq M$. Then
\begin{itemize}
\item[]$Z$ is $\tau$-meager \qquad iff \qquad $Z$ is $(\tau_\#)$-meager.
\end{itemize}
\end{cor}

\begin{lem}\wrlab{lem-one-start-end}
Let $\sigma\in M$. Then both of the following are true:
\begin{itemize}
\item[(i)]The set $[\,\Phi_\R^V(\sigma)\,]\,\,\cap\,\,[\,\{-N\}\times\R^2\,]$ has at most one element.
\item[(ii)]The set $[\,\Phi_\R^V(\sigma)\,]\,\,\cap\,\,[\,\{N\}\times\R^2\,]$ has at most one element.
\end{itemize}
\end{lem}

\begin{proof}
We prove only (ii); the proof of (i) is similar.
Let $w',w''\in\R^2$ and assume that $(N,w'),(N,w'')\in\Phi_\R^V(\sigma)$.
We wish to show that $w'=w''$.

Choose $t',t''\in\R$ such that $(N,w')=\Phi_{t'}^V(\sigma)$ and $(N,w'')=\Phi_{t''}^V(\sigma)$.
Fix $t\in\R$ such that $t\ge t'$ and $t\ge t''$.
Since $V=E$ on $R_N$,
we get both $\Pi(\Phi_t^V(\sigma))=w'$ and $\Pi(\Phi_t^V(\sigma))=w''$.
Then $w'=w''$, as desired.
\end{proof}

\begin{lem}\wrlab{lem-Q-facts}
Let $Q:=M\cap S_N$. Then all of the following are true:
\begin{itemize}
\item[(i)]$Q^\circ=Q\cap(M^\circ)$.
\item[(ii)]$Q^\circ$ is $\tau$-open and $\tau$-dense in $Q$.
\item[(iii)]$Q^\circ$ is $\tau$-open in $M$.
\item[(iv)]$\tau|(Q^\circ)=(\tau_\#)|(Q^\circ)$.
\end{itemize}
\end{lem}

\begin{proof}
Since $S_N$ is $(\tau_\#)$-open in $\R^3$,
we get $(M\cap S_N)^\circ=(M^\circ)\cap S_N$.
Since $M^\circ\subseteq M$, we have $M^\circ=(M^\circ)\cap M$.
Then
$$Q^\circ\,=\,(M\cap S_N)^\circ\,=\,(M^\circ)\cap S_N\,=\,(M^\circ)\cap M\cap S_N\,=\,(M^\circ)\cap Q,$$
proving (i).
From the definition of displayed system,
we know both that $Q$~is $\tau$-open in $M$
and that $M^\circ$~is $\tau$-open and $\tau$-dense in $M$.
Then the intersection $Q\cap(M^\circ)$ is $\tau$-open and $\tau$-dense in $Q$,
and so (ii) follows from (i).
By (ii), $Q^\circ$~is $\tau$-open in $Q$, so, as $Q$ is $\tau$-open in $M$,
it follows that $Q^\circ$~is $\tau$-open in $M$, proving (iii).
From the definition of displayed system,
we know that $\tau|(M^\circ)=(\tau_\#)|(M^\circ)$.
So, because $Q^\circ\subseteq M^\circ$,
we conclude that $\tau|(Q^\circ)=(\tau_\#)|(Q^\circ)$, proving (iv).
\end{proof}

\begin{lem}\wrlab{lem-flatpt-density}
Let $Q:=M\cap S_N$.
Assume that $D$ is generically flat.
Then $[\scrf(D)]\cap[Q^\circ]$
is $(\tau_\#)$-dense in $Q^\circ$.
\end{lem}

\begin{proof}
By \lref{lem-Q-facts}(iii--iv),
$Q^\circ$ is $\tau$-open in $M$ and
$\tau|(Q^\circ)\!=\!(\tau_\#)|(Q^\circ)$.

Since $D$ is generically flat, it follows that
$\scrf(D)$ is $\tau$-comeager in~$M$,
so $[\scrf(D)]\cap[Q^\circ]$ is $\tau$-comeager in $Q^\circ$.
So, because $\tau|(Q^\circ)=(\tau_\#)|(Q^\circ)$,
we see that $[\scrf(D)]\cap[Q^\circ]$ is $(\tau_\#)$-comeager in $Q^\circ$.
Then, by the Baire Category Theorem,
$[\scrf(D)]\cap[Q^\circ]$ is $(\tau_\#)$-dense in $Q^\circ$.
\end{proof}

\begin{lem}\wrlab{lem-periodic-contained}
Let $\sigma\in M$ and assume, for some $T\in\R\backslash\{0\}$,
that $\Phi_T^V(\sigma)=\sigma$.
Then $\Phi_\R^V(\sigma)\subseteq S_N$.
\end{lem}

\begin{proof}
Since $V=E$ on $L_N\cup R_N$,
we see that $[\Phi_\R^V(\sigma)]\cap[L_N\cup R_N]=\emptyset$.
Then $\Phi_\R^V(\sigma)\subseteq[\R^3]\backslash[L_N\cup R_N]\subseteq S_N$.
\end{proof}

\section{The iteration\wrlab{sect-iteration}}

There is a graphic at the end of the paper, following the
references, which should help the reader in understanding the
following lemma.

\begin{lem}\wrlab{lem-iteration-lem}
Let $D=(N,M,\tau,\scra,V)$ be a generically flat displayed system.
Let $\sigma_0\in\scrf(D)$.
Let $k\ge1$ be an integer.
Then there exists a displayed system
$D'=(N',M',\tau',\scra',V')$ such that
\begin{itemize}
\item[(a)] \quad $D'$ is an extension of $D$,
\item[(b)] \quad $D'$ is generically flat,
\item[(c)] \quad $(V',\sigma_0)$ is periodic to order $k$ \qquad and
\item[(d)] \quad $\Phi_\R^{V'}(\sigma_0)\quad\subseteq\quad S_{N'}$.
\end{itemize}
\end{lem}

\begin{proof}
By definition of a displayed system,
$\tau|(M^\circ)=(\tau_\#)|(M^\circ)$
and $\scra|(M^\circ)=(\scra_\#)|(M^\circ)$.
Give $M^\circ$ this common topology and maximal $C^\omega$ atlas.
Give $\R^2$ its standard topology and $C^\omega$ maximal atlas.
Give every open subset of $\R^2$ its standard topology and $C^\omega$ maximal atlas.

Let $H_0^@$ be the $C^\omega$ vector field on $\R^2$ represented by the linear map
$(y,z)\mapsto(y,-z):\R^2\to\R^2$.
Then $H_0^@$ is complete and vanishes at~$(0,0)$.
Moreover, for all $t\in\R$,
$\Phi_t^{H_0^@}(y,z)=(e^ty,e^{-t}z)$.
We let $C_0^@:=\{(y,z)\in\R^2\,|\,y^2+z^2=1\}$ denote the circle
of radius $1$ in $\R^2$ centered at $(0,0)$.
Define $\zeta_0^@:\R^2\to\R$ by $\zeta_0^@(y,z)=y^2+z^2-1$.
Then the function $\zeta_0^@$ is $C^\omega$ and vanishes on and only on~$C_0^@$.
We let $B_0^@:=\{(y,z)\in\R^2\,|\,y^2+z^2\le1\}$ be the closed disk
of radius $1$ in~$\R^2$ centered at $(0,0)$.
Let $Z_H^@:=\R\times\{0\}\subseteq\R^2$ and let $Z_V^@:=\{0\}\times\R\subseteq\R^2$.
Then $C_0^@$, $Z_H^@$ and $Z_V^@$ are all nowhere dense in~$\R^2$.
The sets $\{(0,0)\}$, $Z_H^@$ and $Z_V^@$ are all invariant under the flow of $H_0^@$.
For all $v\in(\R^2)\backslash(Z_V^@)$,
we have: $\Phi_t^{H_0^@}(v)$ leaves compact sets in~$\R^2$, as $t\to\infty$.
For all $v\in(\R^2)\backslash(Z_H^@)$,
we have: $\Phi_{-t}^{H_0^@}(v)$ leaves compact sets in $\R^2$, as $t\to\infty$.

Let $q_0:\R^2\to\R$ be defined by $q_0(y,z)=1-[\exp(-y^2-z^2)]$.
Then $q_0$~is $C^\omega$ and vanishes to order $1$ at $(0,0)$.
Moreover, for all $v\in\R^2\backslash\{(0,0)\}$, we have $0<q_0(v)<1$.
Let $H_*^@:=q_0^kH_0^@$.
Then $H_*^@$~is complete and $H_*^@$ vanishes to order~$2k$ at $(0,0)$.

Let $\psi^@:=\Phi_1^{H_*^@}:\R^2\to\R^2$.
Then $\psi^@:\R^2\to\R^2$ is a $C^\omega$ diffeomorphism
which agrees with the identity $\Id_{\R^2}:\R^2\to\R^2$ to order~$2k$ at~$(0,0)$.
Also, $\psi^@(0,0)=(0,0)$ and $\psi^@(Z_H^@)=Z_H^@$ and $\psi^@(Z_V^@)=Z_V^@$.
For all $v\in(\R^2)\backslash(Z_V^@)$,
we have: $(\psi^@)^m(v)$ leaves compact sets in $\R^2$, as $m\to\infty$.
For all $v\in(\R^2)\backslash(Z_H^@)$,
we have: $(\psi^@)^{-m}(v)$ leaves compact sets in $\R^2$, as $m\to\infty$.

As $\sigma_0\in\scrf(D)$, fix $w_0\in\R^2$ such that
$(-N,w_0),(N,w_0)\in\Phi_\R^V(\sigma_0)$.
We define a translation $\scrt:\R^2\to\R^2$ by $\scrt(v)=v+w_0$.
We then define $\psi:=\scrt\circ(\psi^@)\circ(\scrt^{-1}):\R^2\to\R^2$.
Then $\psi:\R^2\to\R^2$ is a $C^\omega$~diffeomorphism
that agrees with $\Id_{\R^2}:\R^2\to\R^2$ to order $2k$ at~$w_0$.
Then $\psi(w_0)=w_0$.
Let $Z_H:=\scrt(Z_H^@)$ and $Z_V:=\scrt(Z_V^@)$.
Then we have both $\psi(Z_H)=Z_H$ and $\psi(Z_V)=Z_V$.
For all $v\in(\R^2)\backslash(Z_V)$,
we have: $\psi^m(v)$ leaves compact sets in $\R^2$, as $m\to\infty$.
For all $v\in(\R^2)\backslash(Z_H)$,
we have: $\psi^{-m}(v)$ leaves compact sets in $\R^2$, as $m\to\infty$.

We define $B_0:=\scrt(B_0^@)$ and $C_0:=\scrt(C_0^@)$.
Then $w_0\in(B_0)\backslash(C_0)$.
Let $\zeta_0:=(\zeta_0^@)\circ(\scrt^{-1}):\R^2\to\R$.
Then $\zeta_0$ is $C^\omega$ and vanishes on and only on~$C_0$.
Let $B_*:=\psi(B_0)$ and $C_*:=\psi(C_0)$.
Then $w_0\in(B_*)\backslash(C_*)$.
Let $\zeta_*:=(\zeta_0)\circ(\psi^{-1}):\R^2\to\R$.
Then $\zeta_*$ is $C^\omega$ and vanishes on and only on~$C_*$.
The sets $C_0$, $C_*$, $Z_H$ and $Z_V$ are all nowhere dense in $\R^2$.

Define $\alpha,\beta,\gamma:\R\to\R$ and $f,g,h:\R^3\to\R$ by
\begin{eqnarray*}
\alpha(x)&=&(x+N+7)(x+N+6),\\
\beta(x)&=&(x-N-2)(x-N-3)(x-N-4)(x-N-5),\\
\gamma(x)&=&(x-N-6)(x-N-7),\\
f(x,y,z)\,\,&=&\,\,1\,\,-\,\,[\,\exp\,(\,-\,[\alpha(x)]^2\,-\,[\zeta_*(y,z)]^2\,)\,],\\
g(x,y,z)\,\,&=&\,\,1\,\,-\,\,[\,\exp\,(\,-\,[\beta(x)]^2\,-\,[\zeta_0(y,z)]^2\,)\,]\quad\hbox{and}\\
h(x,y,z)\,\,&=&\,\,1\,\,-\,\,[\,\exp\,(\,-\,[\gamma(x)]^2\,-\,[\zeta_*(y,z)]^2\,)\,].
\end{eqnarray*}
Then $f,g,h$ are $C^\omega$ and $0\le f<1$ and $0\le g<1$ and $0\le h<1$.
Also,
\begin{itemize}
\item$f$ vanishes on and only on $\{-N-7,-N-6\}\times C_*$,
\item$g$ vanishes on and only on $\{N+2,N+3,N+4,N+5\}\times C_0$,
\item$h$ vanishes on and only on $\{N+6,N+7\}\times C_*$.
\end{itemize}

As $(-N,w_0),(N,w_0)\in\Phi_\R^V(\sigma_0)$,
we get $(N,w_0)\in\Phi_\R^V(-N,w_0)$.
Fix $t_0\in\R$ such that $(N,w_0)=\Phi_{t_0}^{V}(-N,w_0)$.
Since $V=E$ on $L_N$, we get $t_0>0$.
We have $\pi(\Phi_{t_0}^{V}(-N,w_0))=\pi(N,w)=N$.
Choose an open neighborhood $W_0$ in~$\R^2$ of $w_0$
such that, for all $w\in W_0$,
we have $\pi(\Phi_{t_0}^{V}(-N,w))\in(N-0.5,N+0.5)$.
Let $J:=(t_0-0.5,t_0+0.5)\subseteq\R$.
Then, for all $(t,w)\in J\times W_0$,
we have $\pi(\Phi_t^{V}(-N,w))\in(N-1,N+1)$,
so $\Phi_t^{V}(-N,w)\in(N-1,N+1)\times\R^2\subseteq R_N\subseteq M^\circ$.
Moreover,
$$(t,w)\quad\mapsto\quad\Phi_t^V(-N,w)\quad:\quad J\,\times\,W_0\quad\to\quad M^\circ$$
is $C^\omega$.
Then $(t,w)\mapsto\pi(\Phi_t^V(-N,w)):J\times W_0\to\R$ is $C^\omega$.
By the Implicit Function Theorem,
fix an open neighborhood $W_1$ in~$W_0$ of $w_0$
and a $C^\omega$ function $\lambda_1:W_1\to(0,\infty)$
such that $\lambda_1(w_0)=t_0$ and such that, for all $w\in W_1$,
we have $\pi(\Phi_{\lambda_1(w)}^{V}(-N,w))=N$.

Let $W:=W_1\cap[(B_0)\backslash(C_0)]\cap[(\R^2)\backslash(C_*)]$.
Then $w_0\in W\subseteq(B_0)\backslash(C_0)$ and $W$~is open in $\R^2$.
We have $W\cap C_*=\emptyset$,
and so we conclude that $f\ne0$ on $[-N-6,-N-5]\times W$.
Also, we have $W\cap C_0=\emptyset$,
and so we conclude that $g\ne0$ on $[N+1,N+2]\times W$.
Also, we have that $h\ne0$ on $[N+1,N+2]\times W$.
Then, using the Implicit Function Theorem,
we define $C^\omega$~functions $\lambda_2,\lambda_3:W\to(0,\infty)$ by:
for all $w\in W$,
\begin{eqnarray*}
\Phi_{\lambda_2(w)}^{fE}\,(\,-N-6\,,w\,)&=&(\,-N-5\,,\,w\,)\qquad\hbox{and}\\
\Phi_{\lambda_3(w)}^{ghE}\,(\,N+1\,,\,w\,)&=&(\,N+2\,,\,w\,).
\end{eqnarray*}

Let $T_0:=[\lambda_1(w_0)]+[\lambda_2(w_0)]+[\lambda_3(w_0)]$.
Then, because we have $\lambda_1(w_0),\lambda_2(w_0),\lambda_3(w_0)\in(0,\infty)$,
it follows that $T_0>0$.
Fix an integer~$T$ such that $T_0+2\le T\le T_0+3$.
Then $2<T$, so $T\ne0$.
Define $\lambda:=T-\lambda_1-\lambda_2-\lambda_3:W\to\R$.
Then $\lambda(w_0)=T-T_0$, so $2\le\lambda(w_0)\le3$.
By \lref{lem-match-jet},
fix a $C^\omega$ function $\lambda_0:\R^2\to(1,4)$
such that $\lambda_0$ agrees with $\lambda$ to order $k$ at~$w_0$.
Define $\Lambda:=\lambda_0+\lambda_1+\lambda_2+\lambda_3:W\to\R$.
Then $\Lambda$ agrees with the constant~$T$ to order~$k$ at $w_0$.
Let $N':=N+10$,
\vskip.15in
\begin{itemize}
\item[]$I_1:=(\,-N-8\,,\,-N-5\,]$, \qquad\qquad$I_2:=[\,N+1\,,\,N+8\,)$,
\vskip.17in
\item[]$I_3:=(\,-\infty\,,\,-N-9]$, \qquad\qquad\qquad$I_4:=\,[N+9\,,\,\infty)$,
\end{itemize}
\vskip.1in
\begin{eqnarray*}
X_1&:=&[\,-N-7\,,\,-N-6\,]\,\,\times\,\,B_*,\\
Y_1&:=&\{\,-N-7\,,\,-N-6\,\}\,\,\times\,\,(B_*\backslash C_*),
\end{eqnarray*}
\vskip-.18in
\begin{eqnarray*}
X_2&:=&[\,N+6\,,\,N+7\,]\,\,\times\,\,B_*,\\
Y_2&:=&\{\,N+6\,,\,N+7\,\}\,\,\times\,\,(B_*\backslash C_*),
\end{eqnarray*}
\vskip-.14in
$$X_3\,:=\,[\,N+2\,,\,N+3\,]\,\,\times\,\,B_0,$$
\vskip-.14in
$$X_4\,:=\,[\,N+4\,,\,N+5\,]\,\,\times\,\,B_0,$$
\vskip-.12in
\begin{eqnarray*}
M'_0&:=&M\,\,\cap\,\,S_N,\\
M'_1&:=&[\,\,(I_1\,\,\times\,\,\R^2)\,\,\backslash\,\,X_1\,\,]\,\,\cup\,\,Y_1,\\
M'_2&:=&[\,\,(I_2\,\,\times\,\,\R^2)\,\,\backslash\,\,(X_2\cup X_3\cup X_4)\,\,]\,\,\cup\,\,Y_2,\\
M'_3&:=&I_3\,\,\times\,\,\R^2,\\
M'_4&:=&I_4\,\,\times\,\,\R^2\qquad\qquad\hbox{and}\\
M'_5&:=&\{\,(x,v)\,\in\,\R\times\R^2\,\,|\,\,-N-[\lambda_0(v)]\,<\,x\,\le\,-N\,\}.
\end{eqnarray*}
\vskip.1in
Let $M':=M'_0\cup M'_1\cup M'_2\cup M'_3\cup M'_4\cup M'_5$.
For all integers $j\in[1,4]$, we have $M'_j\not\subseteq(M')^\circ$.
By contrast, we have $M'_5\subseteq(M')^\circ$.
In fact, we have
$(M')^\circ=[(M'_0)^\circ]\cup[(M'_1)^\circ]\cup[(M'_2)^\circ]\cup[(M'_3)^\circ]\cup[(M'_4)^\circ]\cup[M'_5]$.

Define $\rho_1,\rho_2,\rho_3,\rho_4:(-0.5,0.5)\times\R^2\to M'$ by
\begin{eqnarray*}
\rho_1(t,v)&=&\begin{cases}
\Phi_t^E(-N-9,v),&\hbox{if }-0.5<t\le0,\\
\Phi_t^{fE}(-N-8,v),&\hbox{if }0<t<0.5,
\end{cases}\\
\rho_2(t,v)&=&\begin{cases}
\Phi_t^{fE}(-N-5,v),&\hbox{if }-0.5<t\le0,\\
\Phi_t^E(-N-[\lambda_0(v)],v),&\hbox{if }0<t<0.5,
\end{cases}\\
\rho_3(t,v)&=&\begin{cases}
\Phi_t^E(N,v),&\hbox{if }-0.5<t<0,\\
\Phi_t^{ghE}(N+1,v),&\hbox{if }0\le t<0.5,
\end{cases}\\
\rho_4(t,v)&=&\begin{cases}
\Phi_t^{ghE}(N+8,v),&\hbox{if }-0.5<t<0,\\
\Phi_t^E(N+9,v),&\hbox{if }0\le t<0.5.
\end{cases}
\end{eqnarray*}
Define $\rho_5,\rho_6,\rho_7,\rho_8:(-0.5,0.5)\times(B_0\backslash C_0)\to M'$ by
\begin{eqnarray*}
\rho_5(t,v)&=&\begin{cases}
\Phi_t^{fE}(-N-7,\psi(v)),&\hbox{if }-0.5<t\le0,\\
\Phi_t^{ghE}(N+3,v),&\hbox{if }0<t<0.5,
\end{cases}\\
\rho_6(t,v)&=&\begin{cases}
\Phi_t^{ghE}(N+2,v),&\hbox{if }-0.5<t<0,\\
\Phi_t^{fE}(-N-6,\psi(v)),&\hbox{if }0\le t<0.5,
\end{cases}\\
\rho_7(t,v)&=&\begin{cases}
\Phi_t^{ghE}(N+4,v),&\hbox{if }-0.5<t<0,\\
\Phi_t^{ghE}(N+7,\psi(v)),&\hbox{if }0\le t<0.5,
\end{cases}\\
\rho_8(t,v)&=&\begin{cases}
\Phi_t^{ghE}(N+6,\psi(v)),&\hbox{if }-0.5<t\le0,\\
\Phi_t^{ghE}(N+5,v),&\hbox{if }0<t<0.5.
\end{cases}
\end{eqnarray*}
For all integers $j\in[1,8]$, let
$\tau'_j:=\{\rho_j(U)\,|\,U\in\tau_\#\hbox{ and }U\subseteq\dom[\rho_j]\}$.
Let
$\tau'_0:=([\tau_\#]|[(M')^\circ])\cup(\tau|[M\cap S_N])\cup\tau'_1\cup\cdots\cup\tau'_8$.
Then $\tau'_0$ is a basis for a manifold topology $\tau'$ on~$M'$.
The set $(M')^\circ$ is then $\tau'$-open and $\tau'$-dense in $M'$.
For all integers $j\in[1,8]$, 
the map $\rho_j$ is injective and $\dom[\rho_j]\subseteq\R^3$;
let $R_j:=\im[\rho_j]\in\tau'_j\subseteq\tau'$ and
let $\kappa_j:R_j\to\R^3$ be the inverse of $\rho_j$.
Let
$\scra'_0:=([\scra_\#]|[(M')^\circ])\cup(\scra|[M\cap S_N])
\cup\{\kappa_1,\ldots,\kappa_8\}$.
Then $\scra'_0$~is a $C^\omega$ atlas on $(M',\tau')$.
Let $\scra'$ be the unique maximal $C^\omega$~atlas on~$(M',\tau')$
such that $\scra'_0\subseteq\scra'$.
Then $\tau'|[(M')^\circ]=[\tau_\#]|[(M')^\circ]$
and $\scra'|[(M')^\circ]=[\scra_\#]|[(M')^\circ]$.
Give $(M')^\circ$ this common topology and maximal $C^\omega$ atlas.
Define a $C^\omega$ vector field $V'_\circ$ on $(M')^\circ$ by
$$V'_\circ=
\begin{cases}
V&\hbox{on }[(M'_0)^\circ]\cup[(M'_3)^\circ]\cup[(M'_4)^\circ]\cup[M'_5],\cr
fE&\hbox{on }(M'_1)^\circ,\cr
ghE&\hbox{on }(M'_2)^\circ.
\end{cases}$$
By construction of $\tau'$ and $\scra'$,
let $V'$ be the unique $C^\omega$ vector field on~$(M',\tau',\scra')$
such that $V'|[(M')^\circ]=V'_\circ$.
Then $V'=V$ on $M'_0$.
Also $V'=V=E$ on $[(M'_3)^\circ]\cup[(M'_4)^\circ]\cup[M'_5]\cup[L_N\cap S_N]\cup[R_N\cap S_N]$.
Let $D':=(N',M',\tau',\scra',V')$.
Then $D'$~is a displayed system, and from the construction above,
we know that $D'$ is an extension of~$D$.
Therefore (a)~of \lref{lem-iteration-lem} holds.
It remains to prove (b), (c) and (d).

Let $Z'_0:=[M'_0]\backslash[\scrf(D)]$,
\begin{eqnarray*}
Z'_1&:=&\{N+7.5\}\,\,\times\,\,C_*,\\
Z'_2&:=&\{N+3.5\}\,\,\times\,\,C_0,\\
Z'_3&:=&\{-N-5.5\}\,\,\times\,\,[C_*\cup Z_H],\\
Z'_4&:=&\{N+5.5\}\,\,\times\,\,[C_0\cup Z_V],\\
Z'_5&:=&\{N+1.5\}\,\,\times\,\,[C_0\cup Z_V],\\
Z'_6&:=&\{N+5.5\}\,\,\times\,\,C_*,\\
Z'_7&:=&\{-N'\}\,\,\times\,\,C_*,\\
P_1&:=&\big(\,\,[N+7,N+8)\,\times\,B_*\,\,\big)\quad\cap\quad M',\\
P_2&:=&(N+3,N+4)\,\times\,B_0,\\
P_3&:=&\big(\,\,[-N-6,-N-5)\,\times\,B_*\,\,\big)\quad\cap\quad M',\\
P_4&:=&(N+5,N+6)\,\times\,B_0,\\
P_5&:=&(-N-8,-N-7)\,\times\,\R^2,\\
P_6&:=&[N+1,N+2)\,\times\,\R^2 \qquad\qquad\qquad\qquad \hbox{and}\\
P_7&:=&[N+9,N+10)\,\times\,\R^2.
\end{eqnarray*}
For all integers $j\in[0,7]$,
let $Z_j^*:=\Phi_\R^{V'}(Z'_j)$.
Let $Z^*:=Z_0^*\cup\cdots\cup Z_7^*$.
Let $P_*:=\{\,\mu\in M'\,\,|\,\,
\exists\mu_*\in\Phi_\R^{V'}(\mu)\,\,\hbox{ s.t. }\,\pi(\mu_*)<(\pi(\mu))-0.5\,\}$.
\vskip.02in

{\it Claim 1:}
We have $[P_1]\backslash[Z_1^*]\subseteq P_*$.
{\it Proof of Claim 1:}
Suppose that $s\in[N+7,N+8)\subseteq\R$
and that $v\in B_*\subseteq\R^2$.
Let $\mu:=(s,v)\in\R^3$.
Assume that $\mu\in M'$ and that $\mu\notin Z_1^*$.
We wish to prove that $\mu\in P_*$.
We define $\mu_*:=(N+3.5,\psi^{-1}(v))$.
Then $\pi(\mu_*)=N+3.5<s-0.5$,
so $\pi(\mu_*)<(\pi(\mu))-0.5$,
so it remains to prove that
$\mu_*\in\Phi_\R^{V'}(\mu)$.

We have $\mu\notin Z_1^*$, so $[\Phi_\R^{V'}(\mu)]\cap Z'_1=\emptyset$.
So, as $(N+7.5,v)\in\Phi_\R^{V'}(\mu)$, we get $(N+7.5,v)\notin Z'_1$, so $v\notin C_*$.
Then, by construction of $V'$,
we have $(N+3.5,\psi^{-1}(v))\in\Phi_{(-\infty,0)}^{V'}(\mu)$.
Then $\mu_*=(N+3.5,\psi^{-1}(v))\in\Phi_\R^{V'}(\mu)$,
as desired.
{\it End of proof of Claim 1.}

{\it Claim 2:}
We have $[P_2]\backslash[Z_2^*]\subseteq P_*$.
{\it Proof of Claim 2:}
Suppose that $s\in(N+3,N+4)\subseteq\R$
and that $v\in B_0\subseteq\R^2$.
Let $\mu:=(s,v)\in\R^3$.
Assume that $\mu\notin Z_2^*$.
We wish to prove that $\mu\in P_*$.
We define $\mu_*:=(-N-7,\psi(v))$.
Then we have $\pi(\mu_*)=-N-7<s-0.5$,
so $\pi(\mu_*)<(\pi(\mu))-0.5$,
so it remains to prove that
$\mu_*\in\Phi_\R^{V'}(\mu)$.

We have $\mu\notin Z_2^*$, so $[\Phi_\R^{V'}(\mu)]\cap Z'_2=\emptyset$.
So, as $(N+3.5,v)\in\Phi_\R^{V'}(\mu)$, we get $(N+3.5,v)\notin Z'_2$, so $v\notin C_0$.
Then, by construction of $V'$, we have 
$(-N-7,\psi(v))\in\Phi_{(-\infty,0)}^{V'}(\mu)$.
Then $\mu_*=(-N-7,\psi(v))\in\Phi_\R^{V'}(\mu)$,
as desired
{\it End of proof of Claim 2.}

{\it Claim 3:}
We have $[P_3]\backslash[Z_0^*\cup Z_3^*]\subseteq P_*$.
{\it Proof of Claim 3:}
Suppose that $s\in[-N-6,-N-5)\subseteq\R$
and $v\in B_*\subseteq\R^2$.
Let $\mu:=(s,v)\in\R^3$.
Assume $\mu\in M'$, $\mu\notin Z_0^*$ and $\mu\notin Z_3^*$.
We wish to prove that $\mu\in P_*$.

We have $\mu\notin Z_3^*$, and it follows that $[\Phi_\R^{V'}(\mu)]\cap Z'_3=\emptyset$.
So, as $(-N-5.5,v)\in\Phi_\R^{V'}(\mu)$, we get $(-N-5.5,v)\notin Z'_3$, so $v\notin Z_H$.
Fix an integer $m\ge1$ such that
$$v,\psi^{-1}(v),\ldots,\psi^{-m+1}(v)\in B_*\qquad\hbox{and}\qquad\psi^{-m}(v)\notin B_*.$$
Define $\mu_*:=(-N-7,\psi^{-m}(v))$.
Then $\pi(\mu_*)=-N-7<s-0.5$,
so $\pi(\mu_*)<(\pi(\mu))-0.5$,
so it remains to prove that
$\mu_*\in\Phi_\R^{V'}(\mu)$.

We have $\mu\notin Z_0^*$, so $[\Phi_\R^{V'}(\mu)]\cap Z'_0=\emptyset$,
so $[\Phi_\R^{V'}(\mu)]\cap M'_0\subseteq\scrf(D)$.
Also, as $[\Phi_\R^{V'}(\mu)]\cap Z'_3=\emptyset$,
we get $[\Phi_\R^{V'}(\mu)]\cap[\{-N-5.5\}\times C_*]=\emptyset$.
Then, working in reverse time, starting at $\mu=(s,v)$,
we see, from the construction of $V'$, that all of the following
are elements of $\Phi_{(-\infty,0)}^{V'}(\mu)$:
\begin{itemize}
\item[]$(-N-6,v)$,
\item[]$(N-0.5,\psi^{-1}(v))$, \quad $(-N+0.5,\psi^{-1}(v))$, \quad $(-N-6,\psi^{-1}(v))$,
\item[]$(N-0.5,\psi^{-2}(v))$, \quad $(-N+0.5,\psi^{-2}(v))$, \quad $(-N-6,\psi^{-2}(v))$,
\item[]$\cdots\cdots\cdots\cdots$
\item[]$(N-0.5,\psi^{-m}(v))$, \quad $(-N+0.5,\psi^{-m}(v))$, \quad $(-N-6,\psi^{-m}(v))$,
\item[]$(-N-7,\psi^{-m}(v))$.
\end{itemize}
Then $\mu_*=(-N-7,\psi^{-m}(v))\in\Phi_\R^{V'}(\mu)$.
{\it End of proof of Claim 3.}

{\it Claim 4:}
We have $[P_4]\backslash[Z_4^*]\subseteq P_*$
{\it Proof of Claim 4:}
Suppose that $s\in(N+5,N+6)\subseteq\R$
and that $v\in B_0\subseteq\R^2$.
Let $\mu:=(s,v)\in\R^3$.
Assume that $\mu\notin Z_4^*$.
We wish to prove that $\mu\in P_*$.

We have $\mu\notin Z_4^*$, and it follows that $[\Phi_\R^{V'}(\mu)]\cap Z'_4=\emptyset$.
So, as $(N+5.5,v)\in\Phi_\R^{V'}(\mu)$,
we conclude that $(N+5.5,v)\notin Z'_4$, so $v\notin Z_V$.
Fix an integer $m\ge1$ such that
$$v,\psi(v),\ldots,\psi^{m-1}(v)\in B_0\qquad\hbox{and}\qquad\psi^m(v)\notin B_0.$$
Define $\mu_*:=(N+4,\psi^m(v))$.
Then we have $\pi(\mu_*)=N+4<s-0.5$,
so $\pi(\mu_*)<(\pi(\mu))-0.5$,
so it remains to prove that
$\mu_*\in\Phi_\R^{V'}(\mu)$.

Because $[\Phi_\R^{V'}(\mu)]\cap Z'_4=\emptyset$,
we get $[\Phi_\R^{V'}(\mu)]\cap[\{N+5.5\}\times C_0]=\emptyset$.
Then, working in reverse time, starting at $\mu=(s,v)$,
we see, from the construction of $V'$, that all of the following
are elements of $\Phi_{(-\infty,0)}^{V'}(\mu)$:
\begin{itemize}
\item[]$(N+6,\psi(v))$, \quad $(N+5.5,\psi(v))$,
\item[]$(N+6,\psi^2(v))$, \quad $(N+5.5,\psi^2(v))$,
\item[]$\cdots\cdots\cdots\cdots$
\item[]$(N+6,\psi^m(v))$, \quad $(N+5.5,\psi^m(v))$,
\item[]$(N+4,\psi^m(v))$.
\end{itemize}
Then $\mu_*=(N+4,\psi^m(v))\in\Phi_\R^{V'}(\mu)$.
{\it End of proof of Claim~4.}

{\it Claim 5:}
$[M'_0]\backslash[Z'_0]\subseteq P_*$
{\it Proof of Claim 5:}
Let $\mu\in M'_0$,
and assume that $\mu\notin Z'_0$.
We wish to show that $\mu\in P_*$.

Since $\mu\notin Z'_0=[M'_0]\backslash[\scrf(D)]$
and since $\mu\in M'_0$,
we get $\mu\in\scrf(D)$.
So fix $v\in\R^2$ such that $(-N,v)\in\Phi_\R^V(\mu)$.
Let $\mu_*:=(-N-1,v)$.
Because $\mu\in M'_0\subseteq S_N$,
it follows that $-N<\pi(\mu)$.
Then we have $\pi(\mu_*)=-N-1<(\pi(\mu))-0.5$,
so it remains to show that $\mu_*\in\Phi_\R^{V'}(\mu)$.

From the construction of $V'$,
we see that $(-N-1,v)\in\Phi_\R^{V'}(-N,v)$.
We have $(-N,v)\in\Phi_\R^V(\mu)$, so, since $V'=V$ on $M'_0$,
we conclude that $(-N,v)\in\Phi_\R^{V'}(\mu)$, and, therefore,
that $\Phi_\R^{V'}(-N,v)=\Phi_\R^{V'}(\mu)$.
Then $\mu_*=(-N-1,v)\in\Phi_\R^{V'}(-N,v)=\Phi_\R^{V'}(\mu)$.
{\it End of proof of Claim~5.}

{\it Claim 6:} Let $\mu_0\in M'$.
Assume $\mu_0\notin Z_0^*\cup\cdots\cup Z_4^*$.
Then there exists $\mu_*\in\Phi_\R^{V'}(\mu_0)$
such that $\pi(\mu_*)<(\pi(\mu_0))-0.5$.
{\it Proof of Claim 6:}
By definition of $P_*$, we wish to show that $\mu_0\in P_*$.

Let $P':=P_1\cup\cdots\cup P_4\cup M'_0$.
We have $\mu_0\notin Z_1^*$, $\mu_0\notin Z_2^*$,
$\mu_0\notin Z_0^*\cup Z_3^*$ and $\mu_0\notin Z_4^*$.
Also, $\mu_0\notin Z_0^*\supseteq Z'_0$.
Then, by Claims 1--5, we are done if $\mu_0\in P'$,
so we assume that $\mu_0\in(M')\backslash(P')$.

Let $P'':=P_1\cup\cdots\cup P_7\cup M'_0\cup M'_5$.
By construction of~$V'$, for all~$\mu=(s,v)\in\R\times\R^2\subseteq\R^3$, 
we have:
\begin{eqnarray*}
\mu\in(M')\backslash(P'')&\Rightarrow&(s-0.6,v)\,\in\,\Phi_{(-\infty,0)}^{V'}(\mu),\\
\mu\in P_5&\Rightarrow&(-N-9,v)\,\in\,\Phi_{(-\infty,0)}^{V'}(\mu),\\
\mu\in P_6&\Rightarrow&(N-0.5,v)\,\in\,\Phi_{(-\infty,0)}^{V'}(\mu),\\
\mu\in P_7&\Rightarrow&(N+7.5,v)\,\in\,\Phi_{(-\infty,0)}^{V'}(\mu) \qquad \hbox{and}\\
\mu\in M'_5&\Rightarrow&(-N-5,v)\,\in\,\Phi_{(-\infty,0)}^{V'}(\mu).
\end{eqnarray*}
Therefore $(M')\backslash(P'')$, $P_5$, $P_6$, $P_7$ and $M'_5$ are all subsets of $P_*$.

Then $\mu_0\in(M')\backslash(P')\subseteq[(M')\backslash(P'')]\cup P_5\cup P_6\cup P_7\cup M'_5\subseteq P_*$,
as desired.
{\it End of proof of Claim 6.}

{\it Claim 7:} Let $\mu_0\in M'$ and assume that $\mu_0\notin Z_0^*\cup\cdots\cup Z_4^*$.
Then there exists $\mu_1\in\Phi_\R^{V'}(\mu_0)$
such that $\pi(\mu_1)=-N'$.
{\it Proof of Claim~7:}
By, if necessary, repeatedly applying Claim 6,
we arrive at~$\mu_+\in\Phi_\R^{V'}(\mu_0)$
such that $\pi(\mu_+)\le-N'$.
Let $a:=-(\pi(\mu_+))-N'$.
Let $\mu_1:=\Phi_a^{V'}(\mu_+)$.
Then $\mu_1\in\Phi_\R^{V'}(\mu_+)=\Phi_\R^{V'}(\mu_0)$,
and, since $V'=E$ on $(M'_3)^\circ$,
we conclude that $\pi(\mu_1)=(\pi(\mu_+))+a=-N'$,
as desired.
{\it End of proof of Claim 7.}

{\it Claim 8:} Let $v_1\in(B_0)\backslash(B_*)$ and $\mu_1:=(-N',v_1)$.
Assume that $\mu_1\notin Z_0^*\cup Z_5^*\cup Z_6^*$.
Then $(N',v_1)\in\Phi_{(0,\infty)}^{V'}(\mu_1)$.
{\it Proof of Claim~8:}
We have $\mu_1\notin Z_0^*$, so $[\Phi_\R^{V'}(\mu_1)]\cap Z'_0=\emptyset$,
so $[\Phi_\R^{V'}(\mu_1)]\cap M'_0\subseteq\scrf(D)$.
Working in forward time, starting at $\mu_1=(-N',v_1)$,
we see, from the construction of $V'$, that all of the following
are elements of $\Phi_{(0,\infty)}^{V'}(\mu_1)$:
\begin{itemize}
\item[]$(-N-9,v_1)$, \qquad $(-N-5,v_1)$,
\item[]$(-N+0.5,v_1)$, \qquad $(N-0.5,v_1)$, \qquad $(N+1.5,v_1)$.
\end{itemize}

We have $\mu_1\notin Z_5^*$, and it follows that $[\Phi_\R^{V'}(\mu_1)]\cap Z'_5=\emptyset$.
So, because $(N+1.5,v_1)\in\Phi_\R^{V'}(\mu_1)$, we get $(N+1.5,v_1)\notin Z'_5$, so $v_1\notin Z_V$.
Fix an integer $m\ge1$
such that $v_1,\psi(v_1),\ldots,\psi^{m-1}(v_1)\in B_0$
and $\psi^m(v_1)\notin B_0$.
For all integers $j\in[1,m]$,
we have $\psi^j(v_1)\in\psi(B_0)=B_*$.
Then:
\begin{itemize}
\item$v_1\quad\in\quad(B_0)\backslash(B_*)$,
\item$\psi(v_1)\,\,,\,\,\ldots\,\,,\,\,\psi^{m-1}(v_1)\quad\in\quad B_0\cap B_*$ \qquad\qquad and
\item$\psi^m(v_1)\quad\in\quad(B_*)\backslash(B_0)$.
\end{itemize}
Recall that $[\Phi_\R^{V'}(\mu_1)]\cap M'_0\subseteq\scrf(D)$.
Also, as $[\Phi_\R^{V'}(\mu_1)]\cap Z'_5=\emptyset$,
we see that $[\Phi_\R^{V'}(\mu_1)]\cap[\{N+1.5\}\times C_0]=\emptyset$.
Working in forward time, starting at $(N+1.5,v_1)$,
we see, from the construction of $V'$, that all of the following
are elements of $\Phi_{(0,\infty)}^{V'}(\mu_1)$:
\begin{itemize}
\item[]$(-N-6,\psi(v_1))$, \qquad $(N+1.5,\psi(v_1))$,
\item[]$(-N-6,\psi^2(v_1))$, \qquad $(N+1.5,\psi^2(v_1))$,
\item[]$\cdots\cdots\cdots\cdots$
\item[]$(-N-6,\psi^m(v_1))$, \qquad $(N+1.5,\psi^m(v_1))$,
\item[]$(N+3.5,\psi^m(v_1))$, \qquad $(N+5.5,\psi^m(v_1))$.
\end{itemize}
We have $\mu_1\notin Z_6^*$,
and so $[\Phi_\R^{V'}(\mu_1)]\cap Z'_6=\emptyset$.
That is, we have $[\Phi_\R^{V'}(\mu_1)]\cap[\{N+5.5\}\times C_*]=\emptyset$.
Working in forward time, starting at $(N+5.5,\psi^m(v_1))$,
we see, from the construction of $V'$, that all of the following
are elements of $\Phi_{(0,\infty)}^{V'}(\mu_1)$:
\begin{itemize}
\item[]$(N+6,\psi^m(v_1))$,
\item[]$(N+5.5,\psi^{m-1}(v_1))$, \qquad $(N+6,\psi^{m-1}(v_1))$,
\item[]$(N+5.5,\psi^{m-2}(v_1))$, \qquad $(N+6,\psi^{m-2}(v_1))$,
\item[]$\cdots\cdots\cdots\cdots$
\item[]$(N+5.5,v_1)$, \qquad $(N+6,v_1)$,
\item[]$(N+7.5,v_1)$, \qquad $(N+9,v_1)$, \qquad $(N',v_1)$.
\end{itemize}
In particular, $(N',v_1)\in\Phi_{(0,\infty)}^{V'}(\mu_1)$.
{\it End of proof of Claim 8.}

{\it Claim 9:} Let $v_1\in B_*$ and let $\mu_1:=(-N',v_1)$.
Assume that $\mu_1\notin Z'_7$.
Then $(N',v_1)\in\Phi_{(0,\infty)}^{V'}(\mu_1)$.
{\it Proof of Claim 9:}
As $(-N',v_1)=\mu_1\notin Z'_7$, we get $v_1\notin C_*$,
and so $\psi^{-1}(v_1)\notin\psi^{-1}(C_*)=C_0$.
Working in forward time, starting at $\mu_1=(-N',v_1)$,
we see, from the construction of $V'$, that all of the following
are elements of $\Phi_{(0,\infty)}^{V'}(\mu_1)$:
\begin{itemize}
\item[]$(-N-9,v_1)$, \qquad $(-N-7,v_1)$, \qquad $(N+3.5,\psi^{-1}(v_1))$,
\item[]$(N+7,v_1)$, \qquad $(N+9,v_1)$, \qquad $(N',v_1)$.
\end{itemize}
In particular, $(N',v_1)\in\Phi_{(0,\infty)}^{V'}(\mu_1)$.
{\it End of proof of Claim 9.}

{\it Claim 10:} Let $v_1\in\R^2$ and define $\mu_1:=(-N',v_1)$.
We assume that $\mu_1\notin Z_0^*\cup Z_5^*\cup Z_6^*\cup Z'_7$.
Then $(N',v_1)\in\Phi_{(0,\infty)}^{V'}(\mu_1)$.
{\it Proof of Claim~10:}
By Claim~8, we are done if $v_1\in(B_0)\backslash(B_*)$.
By Claim 9, we are done if $v_1\in B_*$.
Then we assume $v_1\notin[(B_0)\backslash(B_*)]\cup B_*=B_0\cup B_*$.
We have $\mu_1\notin Z_0^*$, so $[\Phi_\R^{V'}(\mu_1)]\cap Z'_0=\emptyset$,
so $[\Phi_\R^{V'}(\mu_1)]\cap M'_0\subseteq\scrf(D)$.
Working in forward time, starting at $\mu_1=(-N',v_1)$,
we see, from the construction of $V'$, that all of the following
are elements of $\Phi_{(0,\infty)}^{V'}(\mu_1)$:
\begin{itemize}
\item[]$(-N-9,v_1)$, \qquad $(-N-5,v_1)$, \qquad $(-N+0.5,v_1)$,
\item[]$(N-0.5,v_1)$, \qquad $(N+1,v_1)$, \qquad $(N+9,v_1)$, \qquad $(N',v_1)$.
\end{itemize}
In particular, $(N',v_1)\in\Phi_{(0,\infty)}^{V'}(\mu_1)$.
{\it End of proof of Claim 10.}

{\it Claim 11:} $[M']\backslash[Z^*]\subseteq\scrf(D')$.
{\it Proof of Claim 11:}
Fix $\mu_0\in M'$.
Assume that $\mu_0\notin Z^*$.
We wish to prove that $\mu_0\in\scrf(D')$.

We have $\mu_0\notin Z^*\supseteq Z_0^*\cup\cdots\cup Z_4^*$.
By Claim 7, fix $\mu_1\in\Phi_\R^{V'}(\mu_0)$
such that $\pi(\mu_1)=-N'$.
Let $v_1:=\Pi(\mu_1)$.
Then $(-N',v_1)=\mu_1\in\Phi_\R^{V'}(\mu_0)$.
By definition of $\scrf(D')$,
it now suffices to prove that
$(N',v_1)\in\Phi_\R^{V'}(\mu_0)$.

The set $Z^*$ is $V'$-invariant.
So, since $\mu_0\notin Z^*$,
we get $\mu_1\notin Z^*$.
Then $\mu_1\notin Z^*\supseteq Z_0^*\cup Z_5^*\cup Z_6^*$
and $\mu_1\notin Z^*\supseteq Z_7^*\supseteq Z'_7$.
Then, by Claim 10, $(N',v_1)\in\Phi_\R^{V'}(\mu_1)$.
So, since $\Phi_\R^{V'}(\mu_1)=\Phi_\R^{V'}(\mu_0)$,
we get $(N',v_1)\in\Phi_\R^{V'}(\mu_0)$, as desired.
{\it End of proof of Claim~11.}

{\it Claim 12:} $Z^*$ is $\tau'$-meager in~$M'$.
{\it Proof of Claim 12:}
Because $D$~is generically flat,
it follows that $M\backslash[\scrf(D)]$ is $\tau$-meager in $M$.
So, as $M'_0\subseteq M$,
we see that $Z'_0$ is $\tau$-meager in~$M$.
Then, by \cref{cor-unambig-nw-dense},
$Z'_0$ is $(\tau_\#)$-meager in~$\R^3$.
So, because $Z'_0\subseteq M'$,
it follows, from \cref{cor-unambig-nw-dense},
that $Z'_0$ is $\tau'$-meager in~$M'$.
So, since the set $\Q$ of rational numbers is countable,
we see that $\Phi_\Q^{V'}(Z'_0)$ is $\tau'$-meager in $M'$.

Let $M''_0:=M\cap([-N+0.5,N-0.5]\times\R^2)$ and
let $Z''_0:=(M''_0)\backslash(\scrf(D))$.
Because $V'=E$ on $L_N\cap S_N$,
we see that $Z'_0=\Phi_{(-0.5,0.5)}^{V'}(Z''_0)$.
Then, because $\Q+(-0.5,0.5)=\R=\R+(-0.5,0.5)$,
we conclude that $\Phi_\Q^{V'}(Z'_0)=\Phi_\R^{V'}(Z''_0)=\Phi_\R^{V'}(Z'_0)$.
Then $\Phi_\Q^{V'}(Z'_0)=\Phi_\R^{V'}(Z'_0)=Z_0^*$.
Then $Z_0^*$ is $\tau'$-meager in $M'$.
It remains to prove: $Z_1^*,\ldots,Z_7^*$ are all $\tau'$-meager in $M'$.
We will only handle $Z_1^*$; proofs for $Z_2^*,\ldots,Z_7^*$ are similar.

Let $Z_1^+:=\Phi_{(-0.5,0.5)}^{V'}(Z'_1)$.
Then $Z_1^+\subseteq(N+7,N+8)\times C_*$,
so $Z_1^+$ is $(\tau_\#)$-nowhere dense in $\R^3$.
So, by \lref{lem-unambig-nw-dense},
$Z_1^+$ is $\tau'$-nowhere dense in $M'$.
We have $\Q+(-0.5,0.5)=\R$,
so $\Phi_\Q^{V'}(Z_1^+)=\Phi_\R^{V'}(Z'_1)$.
That is, $\Phi_\Q^{V'}(Z_1^+)=Z_1^*$.
So, as $\Q$ is countable
and $Z_1^+$~is $\tau'$-nowhere dense in~$M'$,
we see that $Z_1^*$ is $\tau'$-meager in $M'$.
{\it End of proof of Claim~12.}

By Claim 11 and Claim 12,
$D'$ is generically flat, proving (b) of \lref{lem-iteration-lem}.
By \lref{lem-periodic-contained}, (d) follows from (c),
so it only remains to prove (c).
That is, it remains to show: $(V',\sigma_0)$ is periodic to order $k$.

Recall that $(-N,w_0)\in\Phi_\R^V(\sigma_0)$.
Recall that $W$, $W_1$ and $W_0$ are open subsets of $\R^2$,
that $w_0\in W\subseteq W_1\subseteq W_0\subseteq\R^2$
and that $W\subseteq(B_0)\backslash(C_0)$.
Recall the $C^\omega$ functions
$\lambda_0:\R^2\to(1,4)$ and $\lambda_1:W_1\to(0,\infty)$
and $\lambda_2,\lambda_3:W\to(0,\infty)$
and $\Lambda=\lambda_0+\lambda_1+\lambda_2+\lambda_3:W\to\R$.
Recall that,
\begin{itemize}
\item[$(+)$]\qquad for all $w\in W_1$, \qquad $\pi(\,\Phi_{\lambda_1(w)}^V(-N,w)\,)\,\,=\,\,N$.
\end{itemize}
Recall that $T$ is an integer and that $T\ne0$.
Recall that $\Lambda$ agrees with the constant $T$ to order $k$ at~$w_0$.
In particular, we have $T=\Lambda(w_0)$.

{\it Claim 13:} For all $w\in W_1$,
we have $\Phi_{\lambda_1(w)}^{V}(-N,w)=(N,w)$.
{\it Proof of Claim 13:}
Let $U':=(-N,-N+1)\times W_1$ and $U'':=[\scrf(D)]\cap U'$.
Because $\tau|(M^\circ)=(\tau_\#)|(M^\circ)$,
because $U'\subseteq L_N\subseteq M^\circ$ and
because $U'$ is $(\tau_\#)$-open in $\R^3$,
it follows that $U'$~is $\tau$-open in $M$ and that $\tau|(U')=(\tau_\#)|(U')$.
Because $\scrf(D)$ is $\tau$-comeager in $M$,
it follows that $U''$~is $\tau$-comeager in $U'$.
Then $U''$ is $(\tau_\#)$-comeager in $U'$.
So, by the Baire Category Theorem, $U''$ is $(\tau_\#)$-dense in $U'$.
So, since $\Pi(U')=W_1$, we conclude that $\Pi(U'')$ is dense in~$W_1$.
Recall that
$$(t,w)\,\,\mapsto\,\,\Phi_t^V(-N,w)\quad:\quad J\,\times\,W_0\,\,\to\,\,M^\circ$$
is $C^\omega$.
Then $w\mapsto\Pi(\Phi_{\lambda_1(w)}^V(-N,w)):W_1\to\R^2$ is $C^\omega$, hence $C^0$.
Then it suffices to prove,
for all $w\in\Pi(U'')$, that $\Phi_{\lambda_1(w)}^{V}(-N,w)=(N,w)$.
Fix $\sigma\in U''$ and
let $w:=\Pi(\sigma)$.
We wish to prove: $\Phi_{\lambda_1(w)}^{V}(-N,w)=(N,w)$.

Let $x:=\pi(\sigma)$.
We have $(x,w)=\sigma\in U''\subseteq U'\subseteq L_N$.
So, since $V=E$ on~$L_N$,
it follows that $(-N,w)\in\Phi_\R^V(x,w)$.
Then we have $(-N,w)\in\Phi_\R^V(x,w)=\Phi_\R^V(\sigma)$,
and so $\Phi_{\lambda_1(w)}^V(-N,w)\in\Phi_\R^V(\sigma)$.
Also, $w=\Pi(\sigma)\in\Pi(U')=W_1$,
so, by $(+)$, we get $\pi(\Phi_{\lambda_1(w)}^V(-N,w))=N$,
{\it i.e.}, $\Phi_{\lambda_1(w)}^V(-N,w)\in\{N\}\times\R^2$.
So, as $\Phi_{\lambda_1(w)}^V(-N,w)\in\Phi_\R^V(\sigma)$, we get
\begin{itemize}
\item[$(*)$] \qquad $\Phi_{\lambda_1(w)}^V(-N,w)\quad\in\quad[\,\Phi_\R^V(\sigma)\,]\,\,\cap\,\,[\,\{N\}\,\times\,\R^2\,]$.
\end{itemize}
As $\sigma\in U''\subseteq\scrf(D)$,
fix $v\in\R^2$ such that $(-N,v),(N,v)\in\Phi_\R^V(\sigma)$.
Then $(-N,w),(-N,v)\in[\Phi_\R^V(\sigma)]\cap[\{-N\}\times\R^2]$.
So, by \lref{lem-one-start-end}(i), $w=v$.
Then $(N,w)=(N,v)\in[\Phi_\R^V(\sigma)]\cap[\{N\}\times\R^2]$.
So, by $(*)$ and \lref{lem-one-start-end}(ii),
$\Phi_{\lambda_1(w)}^V(-N,w)=(N,w)$.
{\it End of proof of Claim 13.}

{\it Claim 14:} Let $w\in W$.
Then $\Phi_{\Lambda(w)}^{V'}(-N-6,w)=(-N-6,\psi(w))$.
{\it Proof of Claim 14:}
By the definitions of $\lambda_2$ and $\lambda_0$ and $V'$, we have:
\begin{itemize}
\item[(i)] \qquad $\Phi_{\lambda_2(w)}^{V'}\,(\,-N-6\,,\,w\,)\quad=\quad(\,-N-5\,,\,w\,)$ \qquad\qquad and
\item[(ii)] \qquad $\Phi_{\lambda_0(w)}^{V'}\,(\,-N-5\,,\,w\,)\quad=\quad(\,-N\,,w\,)$.
\end{itemize}
Since $w\in W\subseteq W_1$, by Claim 13 and the definition of $V'$, we have:
\begin{itemize}
\item[(iii)] \qquad $\Phi_{\lambda_1(w)}^{V'}\,(\,-N\,,w\,)\quad=\quad(\,N+1\,,\,w\,)$.
\end{itemize}
Since $w\in W\subseteq(B_0)\backslash(C_0)$,
by definitions of $\lambda_3$ and $V'$, we have:
\begin{itemize}
\item[(iv)] \qquad $\Phi_{\lambda_3(w)}^{V'}\,(\,N+1\,,\,w\,)\quad=\quad(\,-N-6\,,\,\psi(w)\,)$.
\end{itemize}
Since $\Lambda(w)=\lambda_3(w)+\lambda_1(w)+\lambda_0(w)+\lambda_2(w)$,
the result follows from (i), (ii), (iii) and (iv).
{\it End of proof of Claim~14.}

Let $U_0:=(0,1)\times W$ and $U_1^*:=(-N-6,-N-5)\times W$.
Then $U_0$~is $(\tau_\#)$-open in~$\R^3$
and $U_1^*$ is open in $(M')^\circ$.
We give $U_0$ the topology $(\tau_\#)|(U_0)$
and we give $U_1^*$ the topology $(\tau_\#)|(U_1^*)=(\tau')|(U_1^*)$.
We give $U_0$ the maximal $C^\omega$ atlas $(\scra_\#)|(U_0)$
and we give $U_1^*$ the maximal $C^\omega$~atlas $(\scra_\#)|(U_1^*)=(\scra')|(U_1^*)$.
We define a map $\Gamma:U_0\to U_1^*$ by
$\Gamma(t,w)=\Phi_t^{V'}(-N-6,w)$.
Let $U_1:=\Gamma(U_0)$.
Then, by the Inverse Function Theorem,
the set $U_1$ is open in $U_1^*$.
We give $U_1$ the topology $(\tau_\#)|(U_1)=(\tau')|(U_1)$
and the maximal $C^\omega$ atlas $(\scra_\#)|(U_1)=(\scra')|(U_1)$.
By the Inverse Function Theorem,
$\Gamma:U_0\to U_1$ is a $C^\omega$ diffeomorphism.

Let $\xi_0:=(0.5,w_0)\in U_0$ and let $\xi_1:=\Gamma(\xi_0)\in U_1$.
By definition of~$\Gamma$, we have $\xi_1=\Phi_{0.5}^{V'}(-N-6,w_0)$.
By construction of $V'$, we see that $(-N-6,w_0)\in\Phi_\R^{V'}(-N,w_0)$.
Therefore $\xi_1\in\Phi_\R^{V'}(-N,w_0)$.
Because $(-N,w_0)\in\Phi_\R^V(\sigma_0)$, and because $V'=V$ on~$M'_0$,
it follows that $(-N,w_0)\in\Phi_\R^{V'}(\sigma_0)$.
Then $\Phi_\R^{V'}(-N,w_0)=\Phi_\R^{V'}(\sigma_0)$.
Then $\xi_1\in\Phi_\R^{V'}(-N,w_0)=\Phi_\R^{V'}(\sigma_0)$,
so it suffices to show that
$(V',\xi_1)$ is periodic to order $k$.
Since $T$~is an integer and $T\ne0$,
it suffices to show that $\Phi_T^{V'}$~agrees with $\Id_M$
to order $k$ at $\xi_1$.

By Claim 14, since $T=\Lambda(w_0)$ and since $\psi(w_0)=w_0$,
we conclude that $\Phi_T^{V'}(-N-6,w_0)=(-N-6,w_0)$.
So, because $\xi_1\in\Phi_\R^{V'}(-N-6,w_0)$,
we get $\Phi_T^{V'}(\xi_1)=\xi_1$.
Let $U'_1$ be an open neighborhood in $U_1$ of $\xi_1$
such that $\Phi_T^{V'}(U'_1)\subseteq U_1$.
Give $U'_1$ the topology $(\tau_\#)|(U'_1)=(\tau')|(U'_1)$
and the maximal $C^\omega$ atlas $(\scra_\#)|(U'_1)=(\scra')|(U'_1)$.
Let $\chi_1:=\Phi_T^{V'}|U'_1:U'_1\to U_1$ be
the restriction of $\Phi_T^{V'}$ to $U'_1$.
Let $\iota_1:U'_1\to U_1$ be the inclusion.
We wish to show that $\chi_1$ agrees with $\iota_1$
to order $k$ at $\xi_1$.

Let $U'_0:=\Gamma^{-1}(U'_1)$.
Then $U'_0$ is open in $U_0$.
Give $U'_0$ the topology $(\tau_\#)|(U'_0)$
and the maximal $C^\omega$ atlas $(\scra_\#)|(U'_0)$.
We have $\xi_0\in U'_0$.
Let $\Gamma':=\Gamma|U'_0:U'_0\to U'_1$ be
the restriction of  $\Gamma$ to $U'_0$.
Then $\Gamma':U'_0\to U'_1$ is a $C^\omega$ diffeomorphism.
Let $\chi_0:=[\Gamma^{-1}]\circ\chi_1\circ\Gamma':U'_0\to U_0$.

Define $F:\R\times W\to\R^3$ by
$$F(t,w)\quad=\quad(\,\,T\,-\,[\Lambda(w)]\,+\,t\,\,,\,\,\psi(w)\,\,).$$
Recall that $\xi_0=(0.5,w_0)$.
Recall that $\Lambda$ agrees with the constant~$T$ to order $k$ at~$w_0$
and that $\psi$ agrees with $\Id_{\R^2}$ to order $2k$ at $w_0$.
Then, using $\tau_\#|(\R\times W)$ and $\scra_\#|(\R\times W)$ on $\R\times W$,
and using $\tau_\#$ and $\scra_\#$ on $\R^3$,
it follows that $F:\R\times W\to\R^3$~agrees with the inclusion $\R\times W\to\R^3$
to order $k$ at $\xi_0$.
In particular, $F(\xi_0)=\xi_0$.
Let $U''_0:=[U'_0]\cap[F^{-1}(U_0)]$.
The $U''_0$ is an open neighborhood in $U'_0$ of $\xi_0$
and $F(U''_0)\subseteq U_0$.

{\it Claim 15:} $F|U''_0=\chi_0|U''_0$.
{\it Proof of Claim 15:}
Let $t\in\R$ and $w\in\R^2$
and assume that $(t,w)\in U''_0$.
We wish to prove that $F(t,w)=\chi_0(t,w)$.

Since $(t,w)\in U''_0\subseteq U'_0$,
it follows that $\Gamma'(t,w)\in U'_1$.
Also, since $(t,w)\in U''_0\subseteq F^{-1}(U_0)$,
it follows that $F(t,w)\in U_0$.
From the definitions of $\Gamma'$ and $\Gamma$, we conclude that
$\Gamma'(t,w)=\Gamma(t,w)=\Phi_t^{V'}(-N-6,w)$.
Because $(t,w)\in U'_0\subseteq U_0=(0,1)\times W$, we get $w\in W$.
So, by Claim~14,
$$\Phi_{\Lambda(w)}^{V'}\,(\,-N-6\,,\,w\,)\quad=\quad(\,-N-6\,,\,\psi(w)\,).$$
We apply $\Phi_t^{V'}$ to this last equation;
since $[\Phi_t^{V'}]\circ[\Phi_{\Lambda(w)}^{V'}]=[\Phi_{\Lambda(w)}^{V'}]\circ[\Phi_t^{V'}]$
and since $\Phi_t^{V'}(-N-6,w)=\Gamma'(t,w)$, we get
$$\Phi_{\Lambda(w)}^{V'}\,(\,\Gamma'(t,w)\,)\quad=\quad\Phi_t^{V'}(\,-N-6\,,\,\psi(w)\,).$$
Applying $\Phi_{T-[\Lambda(w)]}^{V'}$ to this last equation yields
$$\Phi_T^{V'}\,(\,\Gamma'(t,w)\,)\quad=\quad\Phi_{T\,-\,[\Lambda(w)]\,+\,t}^{V'}\,(\,-N-6\,,\,\psi(w)\,),$$
which, by definition of $\chi_1$ and $\Gamma$ and $F$, gives
$\chi_1(\Gamma'(t,w))=\Gamma(F(t,w))$.
Then $\chi_0(t,w)=([\Gamma^{-1}]\circ\chi_1\circ\Gamma')(t,w)=F(t,w)$, as desired.
{\it End of proof of Claim~15.}

Let $\iota_0:U'_0\to U_0$ be the inclusion.
As $F:\R\times W\to\R^3$~agrees
with the inclusion $\R\times W\to\R^3$
to order $k$ at $\xi_0$,
it follows, from Claim~15,
that $\chi_0$~agrees
with $\iota_0$
to order $k$ at $\xi_0$.
So, since $\chi_1=\Gamma\circ\chi_0\circ[(\Gamma')^{-1}]$
and since $\iota_1=\Gamma\circ\iota_0\circ[(\Gamma')^{-1}]$
and since $\Gamma'(\xi_0)=\Gamma(\xi_0)=\xi_1$,
it follows that $\chi_1$ agrees with~$\iota_1$
to order $k$ at $\xi_1$, as desired.
\end{proof}

\section{The counterexample\wrlab{sect-counterex}}

\begin{thm}\wrlab{thm-ctrx}
There exists a $C^\omega$ manifold $M$
and a complete $C^\omega$ vector field $V$ on $M$
and a sequence $\sigma_1,\sigma_2,\ldots$ in $M$
such that $\{\sigma_1,\sigma_2,\ldots\}$ is dense in $M$ and
such that,
\begin{itemize}
\item[]for all integers $k\ge1$, \qquad
$(V,\sigma_k)$ is periodic to order $k$.
\end{itemize}
\end{thm}

\begin{proof}
Let $\{\omega_1,\omega_2,\ldots\}$ be a countable $(\tau_\#)$-dense subset of $\R^3$.
We denote the standard norm on $\R^3$ by $|\,\bullet\,|$.
Let $N_0:=1$.
Let $M_0:=\R^3$.
Let $\tau_0:=\tau_\#$.
Let $\scra_0:=\scra_\#$.
Let $V_0:=E$.
Let $D_0:=(N_0,M_0,\tau_0,\scra_0,V_0)$.
Then $D_0$ is a generically flat displayed system.

Let $J_0:=\N$.
Let $Q_0:=M_0\cap S_{N_0}$.
Then $Q_0=S_{N_0}=S_1$, so $Q_0^\circ=S_1$, and so we have $Q_0^\circ\ne\emptyset$.
Let $j_1:=\min\{j\in J_0\,|\,\omega_j\in Q_0^\circ\}$.
By \lref{lem-flatpt-density}, fix $\sigma_1\in[\scrf(D_0)]\cap Q_0^\circ$ such that
$|\sigma_1-\omega_{j_1}|<1$.
By \lref{lem-iteration-lem}, fix
\begin{itemize}
\item[]a generically flat displayed system $D_1:=(N_1, M_1,\tau_1,\scra_1,V_1)$
\end{itemize}
such that
\begin{itemize}
\item$D_1$ is an extension of $D_0$,
\item$(V_1,\sigma_1)$ is periodic to order $1$ \qquad and
\item$\Phi_\R^{V_1}(\sigma_1)\quad\subseteq\quad S_{N_1}$.
\end{itemize}
Let $J_1:=\N\backslash\{j_1\}$.
Let $Q_1:=M_1\cap S_{N_1}$.
By definition of extension, we have
$M_0\cap S_{N_0}=M_1\cap S_{N_0}$ and $N_0<N_1$.
Then
$$Q_0\quad=\quad M_0\cap S_{N_0}\quad=\quad M_1\cap S_{N_0}\quad\subseteq\quad M_1\cap S_{N_1}\quad=\quad Q_1,$$
so $Q_0^\circ\subseteq Q_1^\circ$.
So, since $Q_0^\circ\ne\emptyset$,
we conclude that $Q_1^\circ\ne\emptyset$.
Let $j_2:=\min\{j\in J_1\,|\,\omega_j\in Q_1^\circ\}$.
By \lref{lem-flatpt-density}, fix $\sigma_2\in[\scrf(D_1)]\cap Q_1^\circ$ such that
$|\sigma_2-\omega_{j_2}|<1/2$.
By \lref{lem-iteration-lem}, fix
\begin{itemize}
\item[]a generically flat displayed system $D_2=(N_2,M_2,\tau_2,\scra_2,V_2)$
\end{itemize}
such that
\begin{itemize}
\item$D_2$ is an extension of $D_1$,
\item$(V_2,\sigma_2)$ is periodic to order $2$ \qquad and
\item$\Phi_\R^{V_2}(\sigma_2)\quad\subseteq\quad S_{N_2}$.
\end{itemize}
Let $J_2:=\N\backslash\{j_1,j_2\}$.
Let $Q_2:=M_2\cap S_{N_2}$.
By definition of extension, we have
$M_1\cap S_{N_1}=M_2\cap S_{N_1}$ and $N_1<N_2$.
Then
$$Q_1\quad=\quad M_1\cap S_{N_1}\quad=\quad M_2\cap S_{N_1}\quad\subseteq\quad M_2\cap S_{N_2}\quad=\quad Q_2,$$
so $Q_1^\circ\subseteq Q_2^\circ$.
So, since $Q_1^\circ\ne\emptyset$,
we conclude that $Q_2^\circ\ne\emptyset$.
Let $j_3:=\min\{j\in J_2\,|\,\omega_j\in Q_2^\circ\}$.
By \lref{lem-flatpt-density}, fix $\sigma_3\in[\scrf(D_2)]\cap Q_2^\circ$ such that
$|\sigma_3-\omega_{j_3}|<1/3$.
By \lref{lem-iteration-lem}, fix
\begin{itemize}
\item[]a generically flat displayed system $D_3=(N_3,M_3,\tau_3,\scra_3,V_3)$
\end{itemize}
such that
\begin{itemize}
\item$D_3$ is an extension of $D_2$,
\item$(V_3,\sigma_3)$ is periodic to order $3$ \qquad and
\item$\Phi_\R^{V_3}(\sigma_3)\quad\subseteq\quad S_{N_3}$.
\end{itemize}
Continue {\it ad infinitum}, obtaining a sequence $D_0,D_1,D_2,\ldots$
of displayed systems and a sequence $\sigma_1,\sigma_2,\ldots$ in $\R^3$.
Also constructed are subsets $Q_0\subseteq Q_1\subseteq Q_2\subseteq\cdots$ of $\R^3$
and subsets $J_0\supseteq J_1\supseteq J_2\supseteq\cdots$ of $\N$.
Also constructed is a sequence $j_1,j_2,\ldots$ of distinct positive integers.

Let $M:=Q_0\cup Q_1\cup Q_2\cup\cdots$.
For all integers $k\ge0$, $D_k$ is a displayed system,
and so $Q_k$ is $\tau_k$-open in $M_k$.
Also, for all integers $k\ge0$, for all integers $k'\ge k+1$,
we know that $D_{k'}$ is an extension of $D_k$,
and it follows that
$Q_k$ is $\tau_{k'}$-open in $M_{k'}$, and, moreover, that
$$\tau_{k'}|Q_k=\tau_k|Q_k,\qquad\scra_{k'}|Q_k=\scra_k|Q_k\qquad\hbox{and}\qquad V_{k'}|Q_k=V_k|Q_k.$$
Let $\tau$ be the unique topology on $M$ such that,
for all integers $k\ge0$, we have $\tau|Q_k=\tau_k|Q_k$.
Let $\scra$ be the unique maximal $C^\omega$ atlas on $(M,\tau)$ such that,
for all integers $k\ge0$, we have $\scra|Q_k=\scra_k|Q_k$.
Give $M$ the topology $\tau$ and the
maximal $C^\omega$ atlas $\scra$.
Let $V$ be the unique $C^\omega$ vector field on $M$ such that,
for all integers $k\ge0$, we have $V|Q_k=V_k|Q_k$.

For all integers $k\ge0$,
$Q_k\in\tau_k|Q_k=\tau|Q_k\subseteq\tau$,
so $Q_k$ is open in~$M$.
For all integers $k\ge1$, $(V_k,\sigma_k)$ is periodic to order~$k$;
so, as $\Phi_\R^{V_k}(\sigma_k)\subseteq M_k\cap S_{N_k}=Q_k$
and as $V=V_k$ on~$Q_k$,
we see that $(V,\sigma_k)$~is periodic to order $k$.
Let $\Sigma:=\{\sigma_1,\sigma_2,\ldots\}$.
It remains to show: $\Sigma$ is dense in~$M$.
Since $M=Q_1\cup Q_2\cup\cdots$,
it suffices to show, for all integers $k\ge1$,
that $\Sigma\cap Q_k$ is $\tau$-dense in $Q_k$.
Fix an integer $k\ge1$.
Since $\tau_k|Q_k=\tau|Q_k$,
we wish to prove that
$\Sigma\cap Q_k$ is $\tau_k$-dense in $Q_k$.

Let $\Sigma':=\Sigma\cap Q_k^\circ$.
By \lref{lem-Q-facts}(ii),
$Q_k^\circ$ is $\tau_k$-dense in $Q_k$,
and it therefore suffices to prove that $\Sigma'$ is $\tau_k$-dense in $Q_k^\circ$.
By \lref{lem-Q-facts}(iv), $\tau_k|(Q_k^\circ)=(\tau_\#)|(Q_k^\circ)$,
so it suffices to show: $\Sigma'$ is $(\tau_\#)$-dense in $Q_k^\circ$.
Fix a nonempty $(\tau_\#)$-open subset $U$ of $Q_k^\circ$.
We wish to prove that $\Sigma'\cap U\ne\emptyset$.

Fix a nonempty $(\tau_\#)$-open subset $U_0$ of $U$ 
and fix $\varepsilon>0$ such that,
\begin{itemize}
\item[$(*)$]for all $\omega\in U_0$,
for all $\sigma\in\R^3$, \qquad
if $|\sigma-\omega|<\varepsilon$, \,\, then $\sigma\in U$.
\end{itemize}
Fix an integer $m_0\ge k$ such that $1/m_0<\varepsilon$.
Let $j':=\max\{j_1,\ldots,j_{m_0}\}$.
By density of $\{\omega_{j'+1},\omega_{j'+2},\ldots\}$ in $\R^3$,
fix an integer $j_*\ge j'+1$
such that $\omega_{j_*}\in U_0$.
Since $j_*>j'$, we have $j_*\notin\{j_1,\ldots,j_{m_0}\}$.

Let $n:=\min\{m\in\N\,|\,(m\ge m_0+1)\hbox{ and }(j_m\ge j_*)\}$.
Then we have $n\ge m_0+1$ and $j_n\ge j_*$.
Also, by minimality of $n$, we conclude that $j_*\notin\{j_{m_0+1},\ldots,j_{n-1}\}$.
So, since $j_*\notin\{j_1,\ldots,j_{m_0}\}$,
it follows that $j_*\in\N\backslash\{j_1,\ldots,j_{n-1}\}$.
That is, $j_*\in J_{n-1}$.
Since $k\le m_0\le n-1$, we have $Q_k\subseteq Q_{n-1}$,
so $Q_k^\circ\subseteq Q_{n-1}^\circ$.
Then
$$\omega_{j_*}\quad\in\quad U_0\quad\subseteq\quad
U\quad\subseteq\quad Q_k^\circ\quad\subseteq\quad Q_{n-1}^\circ.$$
Let $J':=\{j\in J_{n-1}\,|\,\omega_j\in Q_{n-1}^\circ\}$.
By definition of $j_n$, $j_n=\min J'$.
So, since $j_*\in J'$, we get $j_n\le j_*$.
So, as $j_n\ge j_*$, we get $j_n=j_*$.
Then $\omega_{j_n}=\omega_{j_*}\in U_0$.
So, as
$|\sigma_n-\omega_{j_n}|<1/n<1/m_0<\varepsilon$,
we conclude, from~$(*)$, that $\sigma_n\in U$.
So, as $U\subseteq Q_k^\circ$, we get $\sigma_n\in Q_k^\circ$.
So, since $\sigma_n\in\Sigma$, we get $\sigma_n\in\Sigma\cap Q_k^\circ=\Sigma'$.
Then $\sigma_n\in\Sigma'\cap U$,
and so $\Sigma'\cap U\ne\emptyset$.
\end{proof}


\bibliography{list}

\eject
\centerline{\includegraphics[scale=.7]{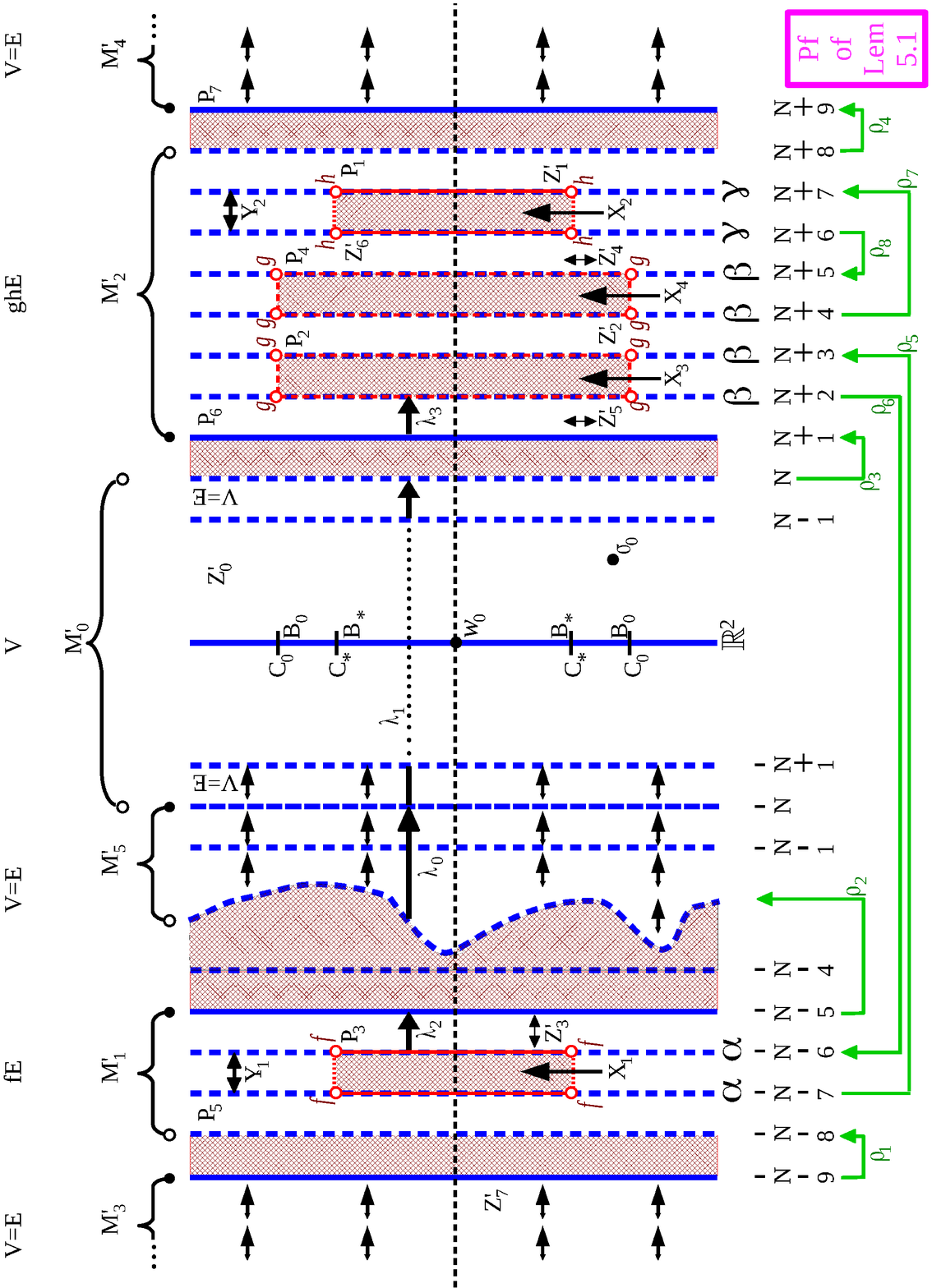}}

\end{document}